\documentclass[12pt]{article}
\usepackage{amsmath,amssymb,amsthm,amsfonts,hyperref,tikz-cd,palatino}
\usepackage{fullpage}

\newtheorem{thm}{Theorem}

\newtheorem{lem}{Lemma}

\newtheorem{rem}{Remark}
\newtheorem{dfn}{Definition}

\newcommand{\rank}{\mathrm{rank}}
\newcommand{\C}{\mathbb{C}}

\title{Spectral lower bounds for the orthogonal and projective ranks of a graph}

\author{Pawel Wocjan\thanks{\texttt{wocjan@cs.ucf.edu}, Department of Computer Science, University of Central Florida, USA}\quad Clive Elphick\thanks{\texttt{clive.elphick@gmail.com}, School of Mathematics, University of Birmingham, Birmingham, UK}}

\begin{document}

\maketitle

\begin{abstract}
The orthogonal rank of a graph $G=(V,E)$ is the smallest dimension $\xi$ such that there exist non-zero column vectors $x_v\in\C^\xi$ for $v\in V$ satisfying the orthogonality condition $x_v^\dagger x_w=0$ for all $vw\in E$.   We prove that many spectral lower bounds for the chromatic number, $\chi$, are also lower bounds for $\xi$. This result complements a previous result by the authors, in which they showed that spectral lower bounds for $\chi$ are also lower bounds for the quantum chromatic number $\chi_q$.  It is known that the quantum chromatic number and the orthogonal rank are incomparable.  

We conclude by proving an inertial  lower bound for the projective rank $\xi_f$, and conjecture that a stronger inertial lower bound for $\xi$ is also a lower bound for $\xi_f$.
\end{abstract}

\section{Introduction}

For any graph $G$, let $V$ denote the set of vertices where $|V| = n$, $E$ denote the set of edges where $|E| = m$, $A$ denote the adjacency matrix, $\chi(G)$ denote the chromatic number, $\omega(G)$ denote the clique number, $\alpha(G)$ denote the independence number, and $\overline{G}$ denote the complement of $G$. Let $\mu_1 \ge \mu_2  \ge \ldots \ge \mu_n$ denote the eigenvalues of $A$. Then, the inertia of $G$ is the ordered triple $(n^+, n^0, n^-)$, where $n^+, n^0$ and $n^-$ are the numbers of positive, zero and negative eigenvalues of $A$, including multiplicities. Note that $\rank(A) = n^+ + n^-$ and $\mathrm{null}(A) = n^0$. A graph is called non-singular if $n^0=0$.

Let $D$ be the diagonal matrix of vertex degrees, and let $L = D - A$ denote the Laplacian of $G$ and $Q = D  + A$ denote the signless Laplacian of $G$. The eigenvalues of $L$ are $\theta_1 \ge \ldots \ge \theta_n = 0$ and the eigenvalues of $Q$ are $\delta_1 \ge \ldots \ge \delta_n$.

Let $\chi_v(G)$ denote the vector chromatic number as defined by Karger \emph{et al} \cite{karger98} and $\chi_{sv}(G)$ denote the strict vector chromatic number.  They proved that $\chi_{sv}(G) = \theta(\overline{G})$, where $\theta$ is the Lov\'asz theta function \cite{lovasz79}.
In this context, we mention that the optimal choice of a weighted adjacency matrix for the Hoffman bound gives a bound that is equal to the strict vector chromatic number, that is,
$1+\max_W \{\mu_1(W \circ A) / |\mu_n(W\circ A)\} = \chi_{sv}(G)$, where the maximum is taken over all Hermitian matrices and $\circ$ denotes the Schur product. Let $\theta^+$ denote Szegedy's \cite{szegedy94} variant of $\theta$. Let $\chi_f(G)$ and $\chi_c(G)$ denote the fractional and circular chromatic numbers and let $\chi_q(G)$ and $\chi_q^{(r)}(G)$ denote the quantum and rank-$r$ quantum chromatic numbers, as defined by Cameron \emph{et al} \cite{cameron07}. 

\begin{dfn}[Orthogonal rank $\xi(G)$]
The orthogonal rank of $G$ is the smallest positive integer $\xi(G)$ such that there exists an orthogonal representation, that is a collection of non-zero column vectors $x_v\in\C^{\xi(G)}$ for $v\in V$ satisfying the orthogonality condition
\begin{equation}
x_v^\dagger x_w = 0
\end{equation}  
for all $vw\in E$. 

The normalized orthogonal rank of $G$ is the smallest positive integer $\xi'(G)$ such that there exists an orthogonal representation, with the added restriction that the entries of each vector must all have the same modulus.
\end{dfn}

Let  $\xi_f(G)$ denote the projective rank which was defined by Man\v{c}inska and Roberson \cite{mancinska16}, who  showed that $\omega(G) \le \xi_f(G) \le \xi(G)$. 
We use the definition of the $r$-fold orthogonal rank $\xi^{[r]}(G)$ due to Hogben et al. in \cite[Section 2.1.]{hogben17} and their results in \cite[Section 2.2.]{hogben17} to provide an equivalent and simpler definition of the projective rank.

\begin{dfn}[$r$-fold orthogonal rank $\xi^{[r]}(G)$ and projective orthogonal rank $\xi_f(G)$]
A $d/r$-representation of $G=(V,E)$ is a collection of rank-$r$ orthogonal projectors $P_v$ for $v\in V$ such that $P_v P_w = 0_d$ for all $vw\in E$. 

The $r$-fold orthogonal rank $\xi^{[r]}(G)$ is defined as follows:
\[
\xi^{[r]}(G) = \min\Big\{ d : G \mbox{ has a $d/r$-representation} \Big\} \,.
\]
The projective rank, $\xi_f(G)$, is defined as follows:
\[
\xi_f(G) = \lim_{r\rightarrow\infty} \frac{\xi^{[r]}(G)}{r}, \mbox{  and this limit exists.}
\]
The projective rank is also called the fractional orthogonal rank. 
\end{dfn}

Clearly, the vectors $x_v\in\C^{\xi(G)}$ of an orthogonal representation correspond to the rank-$1$ orthogonal projectors $P_v = x_v x_v^\dagger\in\C^{\xi(G)\times \xi(G)}$ of a $\xi(G)/1$-representation. It is also clear that $\xi^{[1]}(G)=\xi(G)$.

For $c\in\mathbb{N}$, we use the abbreviation $[c]=\{0,\ldots,c-1\}$.

\begin{dfn}[Vectorial chromatic number $\chi_{\mathrm{vect}}(G)$]\label{def:vectorial}
Paulsen and Todorov \cite{paulsen15} defined the vectorial chromatic number, $\chi_{\mathrm{vect}}(G)$, as follows. Let $G = (V, E)$ be a graph and $c\in\mathbb{N}$. A vectorial $c$-coloring of $G$ is a set of vectors $(x_{v,i} : v \in V, i \in [c])$ in a Hilbert space such that the following conditions are satisfied:

\begin{align}
\langle x_{v,i} , x_{w,j}\rangle \ge 0\,,                                                                                   & \quad v, w \in V ,\,   i,j\in [c] \label{eq:1} \\
\sum_{i\in[c]}  x_{v,i} = \sum_{i\in[c]} x_{w,i}\,,\quad  \Big\| \sum_{i\in[c]} x_{v,i} \Big\| = 1\,, & \quad v, w \in V \label{eq:2} \\
\langle x_{v,i} , x_{v,j} \rangle = 0\,,                                                                                     & \quad v \in V, \, i \not = j\in[c] \label{eq:3} \\
\langle x_{v,i} , x_{w,i} \rangle = 0\,,                                                                                    & \quad vw \in E, \, i\in [c]. \label{eq:4}
\end{align}

The least integer $c$ for which there exists a vectorial c-coloring will be denoted $\chi_{\mathrm{vect}}(G)$ and called the vectorial chromatic number of $G$.

\end{dfn}

Note that $\chi_{\mathrm{vect}}$ differs from $\chi_v$. Cubitt \emph{et al} \cite{cubitt14} (Corollary 16) proved the following (unexpected) equality between a chromatic number and a theta function:
\[
\chi_{\mathrm{vect}}(G) = \lceil \theta^+(\overline{G})\rceil,
\]
and provided an example of  a graph with $\chi_{\mathrm{vect}} < \chi_q$. Paulsen \emph{et al} \cite{paulsen16} (Theorem 7.3) proved that $\theta^+(\overline{G}) \le \xi_f(G) \le \xi(G)$, so $\chi_{\mathrm{vect}}(G) = \lceil \theta^+(\overline{G})\rceil \le \xi(G)$. 

\section{Hierarchies of graph parameters}

There are numerous graph parameters that lie between the clique number and the chromatic number. The following chains of inequalities summarise the relationships between many of them, and combine results in  Cameron \emph{et al} \cite{cameron07},  Man\v{c}inska and Roberson (\cite{mancinska162} and \cite{mancinska16}), Paulsen \emph{et al} \cite{paulsen16} and Elphick and Wocjan \cite{elphick17}. The chains are broken into two parts so the rightmost ends of (\ref{eq:chainA}) and leftmost ends of (\ref{eq:chainB}) coincide.

\begin{equation}\label{eq:chainA}
\begin{tikzcd}
                        &                        &                                                            &                                                     & \lceil \theta^+(\overline{G}) \rceil = \chi_{\mathrm{vect}}(G)  \\
\omega(G) \ar[r] & \chi_v(G) \ar[r] &\chi_{sv}(G) = \theta(\overline{G}) \ar[r]  & \theta^+(\overline{G}) \ar[r] \ar[ru] & \xi_f(G) 
\end{tikzcd}
\end{equation}

\bigskip
\begin{equation}\label{eq:chainB}
\begin{tikzcd}
\chi_{\mathrm{vect}}(G)   \ar[r] \ar[rd]         &  \xi(G)       \ar[to=2-3] \\ 
\xi_f(G ) \ar[to=2-2] \ar[to=1-2] \ar[to=3-2] &  \chi_q(G)  \ar[to=2-3]  & \chi_q^{(1)}(G) \ar[to=2-4] & \xi'(G) \ar[to=2-5] & \lceil \chi_c(G) \rceil = \chi(G) \\
                                                               &  \chi_f(G)     \ar[to=3-3] & \chi_c(G) \ar[to=2-5] \\
\end{tikzcd}
\end{equation}

As illustrated above, Ma\v{n}cinska and Roberson (\cite{mancinska162} and \cite{mancinska16}) demonstrated that $\xi$ and $\chi_q$ are incomparable, as are  $\chi_f$ and $\chi_q$; and also $\chi_f$ and $\xi$. They also proved that $\xi_f$ is a lower bound for $\xi, \chi_q$ and $\chi_f$. Cubitt \emph{et al} \cite{cubitt14} demonstrated that $\chi_{\mathrm{vect}}$ and $\xi_f$ are incomparable.  We can also demonstrate that $\chi_{\mathrm{vect}}$ and $\chi_f$ are incomparable as follows. It is straightforward that for $C_5$, $\chi_{vect} > \chi_f$. However if we consider the disjunctive product $C_5 * K_3$, then from \cite{cubitt14} we have $\chi_{\mathrm{vect}}(C_5 * K_3) \le 7$ but $\chi_f(C_5 * K_3) = 7.5$, because  $\chi_f$ is multiplicative for the disjunctive product. Note that $\xi, \xi', \chi_{\mathrm{vect}}, \chi_q, \chi_q^{(1)}$ are integers, $\chi_f$ is rational but it is unknown if $\xi_f$ is necessarily always rational. These hierarchies of parameters resolve a question raised by Wocjan and Elphick (see Section 2.4 of \cite{wocjan13}) of whether $\chi_v \le \xi'$.

Wocjan and Elphick \cite{wocjan18} proved that many spectral lower bounds for $\chi(G)$ are also lower bounds for $\chi_q(G)$. In this paper we prove that many spectral lower bounds for $\chi(G)$ are also lower bounds for $\xi(G)$. In Theorem~\ref{thm:inertial} we prove an inertial lower bound for $\xi(G)$ by strengthening a proof in \cite{elphick17}. In Theorem~\ref{thm:spectral} we prove several eigenvalue lower bounds for $\xi(G)$ by proving lower bounds for $\chi_{vect}(G)$. We \emph{conjecture} that all of these bounds are also lower bounds for $\xi_f(G)$, and make limited progress in this direction in Theorem~\ref{thm:weaker_inertial}. (See Remark~\ref{rem:progress}.)

\begin{thm}[Inertial lower bound for orthogonal rank]\label{thm:inertial}
Let $\xi(G)$ be the orthogonal rank of a graph $G$ with inertia $(n^+,n^0,n^-)$. Then
\[
1 + \max\left(\frac{n^+}{n^-}, \frac{n^-}{n^+} \right) \le \xi(G).
\]
\end{thm}

\begin{thm}[Eigenvalue lower bounds for vectorial chromatic number]\label{thm:spectral}
Let $\xi(G)$ be the orthogonal rank and $\chi_{vect}(G)$ be the vectorial chromatic number of a graph $G$. Then

\begin{equation}\label{bounds}
 1 + \max\left(\frac{\mu_1}{|\mu_n|} , \frac{2m}{2m - n\delta_n} , \frac{\mu_1}{\mu_1 - \delta_1 + \theta_1} \right) \le \chi_{vect}(G) \le \xi(G).
 \end{equation}
\end{thm}
These bounds, reading from left to right, have been proved to be lower bounds for $\chi(G)$ by  Hoffman \cite{hoffman70}, Lima \emph{et al} \cite{lima11} and Kolotilina \cite{kolotilina11}.

\begin{thm}[Inertial lower bound for projective rank]\label{thm:weaker_inertial}
Let $\xi_f(G)$ be the projective rank of a graph $G$ with inertia $(n^+,n^0,n^-)$.
Then,
\[
1 + \max\left( \frac{n^+}{n^- + n^0}, \frac{n^-}{n^+ + n^0} \right)  \le \xi_f(G)\,.
\]
\end{thm}
In particular, when the graph $G$ is non-singular the lower bounds in Theorems~\ref{thm:inertial} and ~\ref{thm:weaker_inertial} coincide.

\begin{rem}
All results also apply to weighted adjacency matrices $W \circ A$, where $W$ is an arbitrary Hermitian matrix and $\circ$ denotes the Hadamard product (also called Schur product).
\end{rem}

\begin{rem}\label{rem:progress}
The authors have subsequently strengthened Theorem~\ref{thm:spectral} by proving that these three bounds are also lower bounds for the vector chromatic number $\chi_v(G)$ \cite{wocjan182} provided that the weighted adjacency matrices are restricted to have only nonnegative entries. As a result the bounds in Theorem~\ref{thm:spectral} are lower bounds for $\chi_f(G)$ with such weighted adjacency matrices. Theorem~\ref{thm:spectral} in the present paper relies on techniques that are very different from those in \cite{wocjan182}. 
\end{rem}

\section{Proof of the inertial lower bound on the orthogonal rank $\xi(G)$}\label{sec:orthogonal_rank}

Let $f_1,\ldots, f_n\in\C^n$ denote the eigenvectors of unit length corresponding to the eigenvalues $\mu_1 \ge \ldots \ge \mu_n$. Then, $A = B - C$, where

\begin{equation}
B = \sum_{i=1}^{n^+} \mu_i f_i f_i^\dagger\, \quad\mbox{and} \quad
C = \sum_{i=n-n^- +1}^{n} (- \mu_i) f_i f_i^\dagger\,.
\end{equation}

Note that $B$ and $C$ are positive semidefinite and that $\mathrm{rank}(B) = n^+$ and $\mathrm{rank}(C) = n^-$. Let
\[
P^+ = \sum_{i=1}^{n+} f_i f_i^\dagger\,,\quad
P^- = \sum_{i=n-n^- +1}^{n}  f_i f_i^\dagger
\]
denote the orthogonal projectors onto the subspaces spanned by the eigenvectors corresponding to the positive and negative eigenvalues respectively.
Note that $B = P^+ A P^+$ and $C = -P^- A P^-$.

\begin{lem}
\label{lem:rank}
Let $X$ and $Y \in\mathbb{C}^{n\times n}$ be two positive semidefinite matrices satisfying $X \succeq Y$, that is, their difference $X-Y$ is positive semidefinite.  Then, 
\begin{equation}
\mathrm{rank}(X)\ge \mathrm{rank}(Y)\,. 
\end{equation}
\end{lem}
\begin{proof}
Assume to the contrary that $\mathrm{rank}(X) < \mathrm{rank}(Y)$. 
Then, there exists a non-trival vector $v$ in the range of $Y$ that is orthogonal to the range of $X$. Consquently,
\[
v^\dagger (X-Y) v = - v^\dagger Y v < 0
 \] 
 contradicting that $X-Y$ is positive semidefinite.
\end{proof}

\begin{rem}
Let $x_v=(x_v^1,\ldots,x_v^\xi)^T\in\C^\xi$ for $v\in V$ be an orthogonal representation. Note that we may assume that the first entries of these vectors are all equal to $1$, that is,
\[
x_v^1 = 1
\]
for all $v\in V$ for the following reason. If we apply any unitary transformation $U\in\mathbb{C}^{\xi\times\xi}$ to $x_v$ we obtain an equivalent orthogonal representation $y_v = U x_v$.  Clearly, there must exist a unitary matrix $U$ such that the resulting orthogonal representation $y_v = (y_v^1,\ldots,y_v^\xi)^T$ satisfies the condition $y_v^1 \neq 0$ for all $v\in V$
due to a simple parameter counting argument. We can now rescale each vector to additionally achieve $y_v^1 = 1$.  
\end{rem}

We now have all the tools to prove Theorem~\ref{thm:inertial}.

\begin{proof}
Let $x_v=(x_v^1,\ldots,x_v^\xi)^T$ for $v\in V$ be an orthogonal representation satisfying the additional condition $x_v^1=1$ as in the remark above.
We define $\xi$ diagonal matrices 
\[
D_i = \mathrm{diag}(x_v^i : v\in V) \in\mathbb{C}^{n\times n}
\] 
for $i=1,\ldots, \xi$. Due to this construction, we have
\[
\sum_{i=1}^\xi D_i^\dagger A D_i = (s_{vw})\,,\mbox{ with }
s_{vw} = a_{vw} \cdot {x_v^\dagger x_w}  \mbox{ for } v,w\in V\,.
\]
We see that this sum is the zero matrix because all its entries $s_{vw}$ are zero either due to the orthogonality condition of the orthogonal representation $x_v^\dagger x_w=0$ for $vw\in E$ or due to  $a_{vw}=0$ for $vw\not\in E$.
Observe that $D_1=I$ due to the above remark. We obtain 
\begin{equation}\label{eq:conversion}
\sum_{i=2}^\xi D_i^\dagger A D_i = -A.
\end{equation}

Equation~(\ref{eq:conversion}) can be rewritten as
\[
\sum_{i=2}^\xi D_i^\dagger (B - C ) D_i = C - B.
\]
Multiplying both sides by $P^-$ from left and right yields:
\[
P^- \left( \sum_{i=2}^\xi D_i^\dagger (B - C ) D_i \right) P^- = C.
\]
Using that 
\[
P^- \left( \sum_{i=2}^\xi D_i^\dagger C D_i \right) P^-
\]
is positive semidefinite, it follows that
\[
P^- \left( \sum_{i=2}^\xi D_i^\dagger B D_i \right) P^- \succeq C.
\]
Then using that the rank of a sum is less than or equal to the sum of the ranks of the summands, that the rank of a product is less than or equal to the minimum of the ranks of the factors, and Lemma~\ref{lem:rank}, we have that $(\xi-1) n^+ \ge n^-$.
Similarly, $(\xi - 1)n^- \ge n^+$ is obtained by multiplying $(\ref{eq:conversion})$ by $-1$ and repeating the arguments (but multiplying by $P^+$ instead of $P^-$ from the left and right).
\end{proof}

\section{Proof of eigenvalue lower bounds on the orthogonal rank $\xi(G)$}\label{sec:orthogonal_eigenvalue}

We now present a generalization of \cite[Theorem 1]{wocjan18}.

\begin{thm}
Assume that there exists a vectorial $c$-coloring of $G$. Then, there exists a collection of orthogonal projectors $(P_{v,i}\in\C^{d\times d}, v\in V, i\in [c])$ and a unit (column) vector $y\in\C^d$ such that the block-diagonal orthogonal projectors
\[
P_i = \sum_{v\in V} e_v e_v^\dagger \otimes P_{v,i} \in \C^{n\times n} \otimes \C^{d\times d}
\]
satisfy the following three conditions:
\begin{align}
\sum_{i\in [c]} P_i &= I_n \otimes I_d, \label{eq:prop1} \\
\Big( I_n \otimes y y^\dagger \Big) \sum_{i\in [c]} P_i (A\otimes I_d) P_i \Big( I_n \otimes y y^\dagger \Big) &= 0_n \otimes 0_d, \label{eq:prop2} \\
\Big( I_n \otimes y y^\dagger \Big) \sum_{i\in [c]} P_i (E\otimes I_d) P_i \Big( I_n \otimes y y^\dagger \Big) &= E \otimes y y^\dagger, \label{eq:prop3}
\end{align}
where $E\in\C^{d\times d}$ is an arbitrary diagonal matrix.
\end{thm}
\begin{proof}
We now prove condition (\ref{eq:prop1}).
Let $(x_{v,i} : v \in V, i \in [c])$ be a vectorial $c$-coloring of $G$. Conditions (\ref{eq:2}) and (\ref{eq:3}) in Definition~\ref{def:vectorial} imply that there exist orthogonal projectors $P_{v,i}\in\C^{d\times d}$ and a unit (column) vector $y\in\C^d$ such that the $P_{v,i}$ form a resolution of the identity $I_d$
\begin{equation}\label{eq:resolution_block}
\sum_{i\in[c]} P_{v,i} = I_d 
\end{equation}
for all $v\in V$ and
\[
x_{v,i} = P_{v,i} y
\]
for all $v\in V$ and $i\in[c]$. The unit vector $y$ is simply equal to the sum $\sum_{i\in [c]} x_{v,i}$. For $i\in [c]$, $P_{v,i}$ is equal to the orthogonal projector onto the subspace spanned by $x_{v,i}$. Note that if their sum $\sum_{i\in [c]} P_{v,i}$ is not equal to the identity $I_d$, then we can add the projector onto the missing orthogonal complement to, say, $P_{v,0}$. 

We now prove condition (\ref{eq:prop2}).
Let $e_v$ denote the standard basis (column) vectors of $\C^n$ corresponding to the vertices $v\in V$ so that $A=\sum_{v,w\in V} a_{vw} e_v e_w^\dagger$.
For $v\in V, i \in [c]$, the block-diagonal projectors $P_i\in\C^{n\times n}\otimes\C^{d\times d}$ form a resolution of the identity $I_n\otimes I_d$, which follows by applying condition~(\ref{eq:resolution_block}) to each block of these projectors. For $v,w\in V$, we use $v\sim w$ to denote that these vertices are connected. When used in a summation symbol, it means that the sum is taken over all pairs of connected vertices. To abbreviate, we define the projector $\Upsilon=y y^\dagger$.
\begin{align}
& 
\Big( I_n \otimes \Upsilon \Big)
\sum_{i\in [c]} P_i \Big( A\otimes I_d \Big) P_i 
\Big( I_n \otimes \Upsilon \Big) \\
&= 
\Big( I_n \otimes \Upsilon \Big)
\sum_{i\in [c]} \left( \sum_{v\in V} e_v e_v^\dagger \otimes P_{v,i} \right)
\Big( A \otimes I_d \Big)
\left( \sum_{w\in V} e_w e_w^\dagger \otimes P_{w,i} \right)
\Big( I_n \otimes \Upsilon \Big) \\
&=
\Big( I_n \otimes \Upsilon \Big)
\sum_{i\in [c]} \left( \sum_{v,w\in V} a_{v,w} \cdot e_v e_w^\dagger \otimes P_{v,i} P_{w,i} \right)
\Big( I_n \otimes \Upsilon \Big) \\
&=
\Big( I_n \otimes \Upsilon \Big)
\sum_{i\in [c]} \left( \sum_{v\sim w} e_v e_w^\dagger \otimes P_{v,i} P_{w,i} \right)
\Big( I_n \otimes \Upsilon \Big) \\
&=
\sum_{i\in [c]} \left( \sum_{v\sim w} e_v e_w^\dagger \otimes \Upsilon P_{v,i} P_{w,i} \Upsilon \right) \\
&=
\sum_{i\in [c]} \left( \sum_{v\sim w} x_{v,i}^\dagger x_{w,i} \cdot e_v e_w^\dagger \otimes  \Upsilon \right) \\
&=
0_n \otimes 0_d,
\end{align}
where we used
\begin{equation}
\Upsilon P_{v,i} P_{w,i} \Upsilon  = y (y^\dagger P_{v,i}) (P_{w,i} y) y^\dagger =  y (x_{v,i}^\dagger x_{w,i}) y^\dagger = x_{v,i}^\dagger x_{w,i} \cdot \Upsilon
\end{equation}
and (\ref{eq:4}), which states that $x_{v,i}^\dagger x_{w,i}=0$ for all $i\in[c]$ and all $v\sim w$.

Finally, condition~(\ref{eq:prop3}) is proved similarly.
\end{proof}

\begin{lem}
There exists a unitary matrix $U\in\C^{n\times n} \otimes \C^{d \times d}$ such that
\begin{align}
\sum_{\ell\in[c]} P U^\ell (A\otimes I_d) (U^\dagger)^\ell P &= 0_n \otimes 0_d \label{eq:twirling A}\\
\sum_{\ell\in[c]} P U^\ell (E\otimes I_d) (U^\dagger)^\ell P &= c \, E \otimes y y^\dagger\, \label{eq:twirling diag}
\end{align}
for any diagonal matrix $E\in\C^{n\times n}$. 
\end{lem}
\begin{proof}
We now make use of \cite[Lemma 1]{wocjan18} to construct a unitary matrix $U$ from the orthogonal projectors $(P_i : i \in [c])$ such that its powers $U^\ell$ satisfy 
\begin{equation*}
\frac{1}{c} \sum_{\ell\in[c]}U^\ell (X\otimes I_d)(U^\dagger)^\ell = \sum_{i\in[c]} P_i(X\otimes I_d)P_i
\end{equation*} 
for any matrix $X\in\C^{d\times d}$. The left hand side of this equation defines a so-called twirling of the matrix $X \otimes I_d$.
\end{proof}
We now have the tools to prove Theorem~\ref{thm:spectral}.  Note that we did not make use of condition~(\ref{eq:1}).

\subsection{Proof of the Lima  bound in Theorem~\ref{thm:spectral}}
\begin{proof}
The proof is almost identical to the proof for the chromatic number.  
We use the identity $D-Q=-A$. To abbreviate, we set $P=I_n\otimes \Upsilon=I_n\otimes y y^\dagger$. We have:
\begin{eqnarray*}
A \otimes y y^\dagger 
& = &
P ( A\otimes I_d ) P \\
& = &
\sum_{\ell=1}^{c-1} P U^\ell (-A\otimes I_d) (U^\dagger)^{\ell} P \\
& = & 
\sum_{\ell=1}^{c-1} P U^\ell \left( (D-Q)\otimes I_d \right) (U^\dagger)^{\ell} P \\
& = &
(c-1)(D\otimes y y^\dagger) - \sum_{\ell=1}^{c-1} P U^\ell (Q\otimes I_d) (U^\dagger)^\ell P  
\end{eqnarray*}
Define the column vector $v=\frac{1}{\sqrt{n}}(1,1,\ldots,1)^\dagger \otimes y$. Multiply the left and right most sides of the above matrix equation by $v^\dagger$ from the left and by $v$ from the right to obtain
\[
\frac{2m}{n} = v^\dagger (A\otimes y y^\dagger)v = (c-1)\frac{2m}{n} - \sum_{\ell=1}^{c-1} v^\dagger P U^\ell (Q \otimes I_d) (U^\dagger)^\ell P v \le (c-1)\frac{2m}{n} - (c-1)\delta_n\,.
\]
This uses that $v^\dagger (A\otimes y y^\dagger) v = v^\dagger(D\otimes y y^\dagger) v = 2m/n$, which is equal to the sum of all entries of respectively $A$ and $D$ divided by $n$ due to the special form of the vector $v$, and that $v^\dagger P U^\ell (Q\otimes I_d) (U^\dagger)^\ell P v = v^\dagger U^\ell (Q\otimes I_d) (U^\dagger)^\ell v  \ge \lambda_{\min}(Q) = \delta_n$. 
\end{proof}
\subsection{Proof of the Hoffman and Kolotilina bounds in Theorem~\ref{thm:spectral}}
\begin{proof}
Let $E\in\C^{n\times n}$ be an arbitrary diagonal matrix. Using (\ref{eq:twirling A}) and (\ref{eq:twirling diag}), we obtain
\[
\sum_{\ell=1}^{c-1} P U^\ell ( E \otimes I_d - A \otimes I_d ) (U^\dagger)^\ell P = (c-1) E \otimes y y^\dagger + A \otimes y y^\dagger\,.
\]

Using that $\lambda_{\max} (X) \ge \lambda_{\max}(PXP)$ and $\lambda_{\max}(X) + \lambda_{\max}(Y)\ge \lambda_{\max}(X+Y)$ for arbitrary Hermitian matrices $X$ and $Y$, we obtain
\begin{eqnarray*}
\lambda_{\max}(E - A)
& = & 
\lambda_{\max}(E \otimes I_d - A \otimes I_d) \\ 
& \ge &
\lambda_{\max} \left( E \otimes y y^\dagger + \frac{1}{c-1} A \otimes y y^\dagger \right) \\
& = &
\lambda_{\max} \left( E + \frac{1}{c-1} A \right)\,.
\end{eqnarray*}

\cite[Corollary~5]{unified_spectral_bounds} shows that the above eigenvalue bound implies
\[
\lambda_{\max}(E-A) \ge \lambda_{\max}(E+A) - \frac{c-2}{c-1}\lambda_{\max}(A)\,,
\]
or equivalently
\[
c \ge 1 + \frac{\lambda_{\max}(A)}{\lambda_{\max}(A) - \lambda_{\max}(E+A) + \lambda_{\max}(E-A)}\,,
\]
from which the Hoffman and  Kolotilina bounds are obtained by setting $E=0$ and $E= D$, respectively.
\end{proof}

\subsection{Inertial and generalized Hoffman and Kolotilina bounds}

We do not know whether the inertial bound in Theorem~\ref{thm:inertial} or the generalized (multi-eigenvalue) bounds in \cite{unified_spectral_bounds} are also lower bounds for the vectorial chromatic number.  The difficulty seems to be in determining what happens to the \emph{entire} spectrum of the various matrices when they are compressed by $P=I_n\otimes y y^\dagger$. The Kolotilina and Lima bounds only use  the maximal and/or minimal eigenvalues.

\section{Proof of the inertial lower bound on the projective rank $\xi_f(G)$}\label{sec:projective_rank}

We conjecture that for all graphs $G$ the projective rank $\xi_f(G)$ is  lower bounded by
\[
1 + \max\left(\frac{n^+}{n^-}, \frac{n^-}{n^+} \right) \le \xi_f(G)\,.
\]
Unfortunately, we are not able to settle this question by either providing a counterexample or proving  this bound for all graphs.  However, we are able to prove the weaker lower bound in Theorem~\ref{thm:weaker_inertial}.

We derive two lemmas to better organize the proof of Theorem~\ref{thm:weaker_inertial}. 

\begin{lem}\label{lem:simple_rank_bound}
Let $P$ be an orthogonal projector and $X$ a positive semidefinite matrix in $\C^{m\times m}$. Then, we have
\[
\rank(PXP) \ge \rank(P) - \mathrm{null}(X)\,.
\]
\end{lem}
\begin{proof}
There exist positive semidefinite matrices $Y$ and $\Delta$ such that $Y$ has full rank, $\Delta$ has rank $\mathrm{null}(X)$, and $X+\Delta=Y$. Using that $\rank(M+N)\le \rank(M) + \rank(N)$ for arbitrary matrices,
we obtain
\[
\rank(PXP) + \rank(P\Delta P) \ge \rank(PYP)\,.
\]
Using that $\rank(MN)\le\rank(M)$ for arbitrary matrices $M$ and $N$, we obtain
\[
\rank(PXP) \ge \rank(PYP) - \rank(\Delta)\,.
\]
We can write
\[
PYP = (Y^{1/2}P)^\dagger (Y^{1/2}P)\,.  
\]
\end{proof}
Using that $\rank(M^\dagger M)=\rank(M)$ for arbitrary matrices, we obtain 
\[
\rank(PYP) = \rank(Y^{1/2}P) = \rank(P)
\]
because $Y^{1/2}$ has full rank.

\begin{lem}\label{lem:annihilate}
Let $P_v$ be the projectors of a $(d/r)$-orthogonal representation. Define the block diagonal projector
\[
P = \sum_{v\in V} e_v e_v^\dagger \otimes P_v\,.
\]
Then, we have
\[
P(A\otimes I_d) P = 0_n \otimes 0_d
\]
and
\[
\rank(P) = n r\,.
\]
\end{lem}
\begin{proof}
This follows directly from the orthogonality condition $P_v P_w = 0_d$ for all $vw\in E$. The proof is very similar to the proof for the vectorial chromatic number in the previous section.
The projectors $P_v$ have rank $r$ for all $v\in V$ so $\rank(P)=nr$.
\end{proof} 

We are now ready to prove Theorem~\ref{thm:weaker_inertial}.
\begin{proof}
Let $A=B-C$, defined as in Section~\ref{sec:orthogonal_rank}, so $\rank(B)=n^+$ and $\rank(C)=n^-$. Note that  Lemma~\ref{lem:annihilate} implies
\[
P(B\otimes I_d) P = P(C\otimes I_d) P\,.
\]
so that 
\begin{equation}\label{eq:equal}
P(B \otimes I_d)P = \frac{1}{2} P((B+C) \otimes I_d)P\,. 
\end{equation}

Clearly, the rank of the left hand side of (\ref{eq:equal}) is bounded from above by $n^+ d = \rank(B\otimes I_d)$. 

We now bound the rank of the right hand side of (\ref{eq:equal}) from below. Observe that $B+C=|A|$, where $|A|=\sum_{i=1}^n |\mu_i| e_i e_i^\dagger$  and $\mu_i$ and $e_i$ are the eigenvalues and eigenvectors of $A$, respectively.  Clearly, $|A|$ is positive semidefinite, its rank is equal to $\rank(A)=n^+ +n^-$ and its nullity is equal to $\mathrm{null}(A)=n_0$.
Therefore, $|A|\otimes I_d$ is positive semidefinite, its rank is equal to $(n^+ + n^-)d$ and its nullity is equal to $n_0 d$.  We can now apply Lemma~\ref{lem:simple_rank_bound} to obtain
\[
\rank\Big(P \big( |A| \otimes I_d \big) P\Big) \ge \rank(P) - n_0 d = nr - n_0 d\,.
\]
Combining the upper and lower bounds on the ranks, we obtain
\[
n^+ d \ge nr - n_0 d \quad \Longleftrightarrow \quad \frac{d}{r} \ge 1 + \frac{n^-}{n^+ + n^0}\,.
\]
The result
\[
\frac{d}{r} \ge 1 + \frac{n^+}{n^- + n^0}
\]
is obtained by considering $P(C \otimes I_d)P$ on the left hand side of (\ref{eq:equal}). 
\end{proof}

\section{Implications for the projective rank}

The following examples demonstrate that the inertial bound is exact for $\xi_f$  for various classes of graphs. We also use Theorem~\ref{thm:weaker_inertial} to derive the value of $\xi_f$ for some graphs.

\begin{itemize}
\item For \emph{odd cycles}, $C_{2k+1}$ (see \cite{elphick17} and \cite{mancinska16}):
\[
1 + \max\left(\frac{n^+}{n^-} , \frac{n^-}{n^+}\right) = \chi_f = \xi_f = 2 + \frac{1}{k}  ; \mbox{       but  } \chi_{vect} = \chi_q = \xi = \chi = 3.
\]

\item For \emph{Kneser} graphs, $K_{p,k}$ (see \cite{elphick17}, \cite{godsil16}, \cite{paulsen16}, and \cite{haviv18}):
\[
1 + \max\left(\frac{n^+}{n^-} , \frac{n^-}{n^+}\right) = \chi_v = \chi_f = \xi_f = \frac{p}{k}; 
\chi_{vect} = \left\lceil\frac{p}{k}\right\rceil; \mbox{  but  } \xi = \chi = p - 2k + 2. 
\]

\item The \emph{orthogonality} graph, $\Omega(n)$, has vertex set the set of $\pm1-$vectors of length $n$, with two vertices adjacent if they are orthogonal. With $n$ a multiple of 4 (see  \cite[Lemma 4.2 and Theorem 6.4]{mancinska16}):
\[
\chi_{sv} = \chi_{vect} = \xi_f =  \xi = \chi_q = n;  \mbox{   but  } \chi_f \mbox{ and } \chi \mbox{ is exponential in } n.
\]
$\Omega(4)$ has spectrum $(6^2, 0^8, -2^6)$ so when $n = 4$:
\[
1 + \max\left(\frac{n^+}{n^-} , \frac{n^-}{n^+}\right) = 1 + \frac{n^-}{n^+} = 4 = \xi_f = \chi_q,
\]
but for $n > 4$ this  inertial bound is less than $\xi_f$.

\item The \emph{Andr\'asfai} graphs, And(k), are $k$-regular with $(3k-1)$ vertices. It is known (\cite{godsil13} and \cite{godsil06}) that
\[
1 + \max\left(\frac{n^+}{n^-} , \frac{n^-}{n^+}\right) = 1 + \frac{n^+}{n^-} = 1 + \frac{2k - 1}{k} = 3 - \frac{1}{k} = \chi_f 
\]
but 
\[
\chi = \chi_q = \xi = 3.
\]
The Andr\'asfai graphs are non-singular, so using Theorem~\ref{thm:weaker_inertial} and that $\xi_f \le \chi_f$ it follows that $\xi_f = 3 - 1/k$.

\item The \emph{Clebsch} graph on 16 vertices has spectrum $(5^1, 1^{10}, -3^5)$ and $\chi_f = 3.2$ (see \cite{godsil06}). Therefore
\[
1 + \max\left(\frac{n^+}{n^-} , \frac{n^-}{n^+}\right) = 1 + \frac{n^+}{n^-} = 3.2 = \chi_f; \mbox{ but } \chi = \xi = 4.
\]
The Clebsch graph is non-singular, so using Theorem~\ref{thm:weaker_inertial} and that $\xi_f \le \chi_f$, it follows that $\xi_f = 3.2$.

The Clebsch graph is the folded $5$-cube. The folded $7$-cube on $64$ vertices has spectrum $7^1$, $3^{21}$, $-1^{35}$, $-5^7$, so $\chi_f$ for the folded $7$-cube is greater than or equal to $32/11$.
\end{itemize}

More generally, if the inertial bound is exact for the fractional chromatic number of a non-singular graph, then it is also exact for the projective rank. Vertex transitive graphs have $\xi_f \le \chi_f = n/\alpha$, so if a non-singular vertex transitive graph has 

\[
\alpha = \min{(n^+ , n^-)},  \mbox{  then  } \xi_f = \chi_f = \frac{n}{\alpha}.
\]

\section{Conclusion}

We have proved that many lower bounds for $\chi(G)$ are also lower bounds for $\xi(G)$. We have also proved that for  non-singular graphs

\[
1 + \max\left(\frac{n^+}{n^-} , \frac{n^-}{n^+}\right) \le \xi_f(G).
\]
Elphick and Wocjan \cite{elphick17}  proved this lower bound for $\chi_f$ for  non-singular graphs, using a simpler proof technique. 
 
Costello \emph{et al} \cite{costello06} proved that almost all (random) graphs with no isolated vertices are non-singular. This provides limited support for our conjecture that the inertial lower bound for $\xi(G)$ is also a lower bound for $\xi_f(G)$ and consequently for $\chi_f(G)$. 
\section*{Acknowledgements}
This research has been supported in part by NSF Award 1525943.

\end{document}